\documentclass[preprint,12pt,sort&compress,numbers]%,draft]
{elsarticle}

\usepackage{amsmath,amsthm,amsfonts,amssymb,latexsym,mathrsfs,url,enumerate}
\usepackage{graphicx,ifpdf,epstopdf}

\newtheorem{thm}{Theorem}[section]
\newtheorem{lem}[thm]{Lemma}
\newtheorem{coro}[thm]{Corollary}
\newtheorem{prop}[thm]{Proposition}

\newtheorem{conj}[thm]{Conjecture}
\newtheorem*{NI}{Newton Inequality}

\theoremstyle{definition}

\newtheorem{exm}[thm]{Example}
\newtheorem{rem}[thm]{Remark}
\numberwithin{equation}{section}

\newcommand{\lrf}[1]{\lfloor #1\rfloor}

\def\la{\lambda}
\def\sg{\sigma}

\def\sp{\preccurlyeq}

\journal{arXiv}

\usepackage{geometry}
\begin{document}

\begin{frontmatter}

\title{Analytic properties of sextet polynomials of hexagonal systems}
\author[a]{Guanru Li\corref{cor1}}
\ead{li\_guanru@hotmail.com}
\author[b]{Lily Li Liu}
\ead{liulily@qfnu.edu.cn}
\author[a]{Yi Wang\corref{cor2}}
\ead{wangyi@dlut.edu.cn}
\cortext[cor2]{Corresponding author.}

\address[a]{School of Mathematical Sciences, Dalian University of Technology, Dalian 116024, P.R. China}
\address[b]{School of Mathematical Sciences, Qufu Normal University, Qufu 273165, P.R. China}

\begin{abstract}
In this paper we investigate analytic properties of sextet polynomials of hexagonal systems.
For the pyrene chains,
we show that zeros of the sextet polynomials $P_n(x)$ are real,
located in the open interval $(-3-2\sqrt{2},-3+2\sqrt{2})$
and dense in the corresponding closed interval.
We also show that coefficients of $P_n(x)$ are
symmetric, unimodal, log-concave, and asymptotically normal.
For general hexagonal systems,
we show that real zeros of all sextet polynomials are dense in the interval $(-\infty,0]$,
and conjecture that every sextet polynomial has log-concave coefficients.
\end{abstract}


\end{frontmatter}

\section{Introduction}

A {\it benzenoid system} or {\it hexagonal system}
is a finite connected plane graph without cut vertices
in which every interior face is bounded by a regular hexagon of side length $1$.
Topological properties of hexagonal systems is quite important in various quantum mechanical models of
the electronic structure of benzenoid hydrocarbons,
especially in resonance theory,
H\"{u}ckel molecular orbital theory,
Clar's aromatic sextet theory and
the theory of conjugated circuits \cite{Gut83,GC89}.
In order to develop a general algebraic relation
between the application of resonance theory and hexagonal systems,
Hosoya and Yamaguchi \cite{HY75} introduced the concept of the sextet polynomial,
which is the first genuine mathematical object in the aromatic sextet theory,
and demonstrated various resonance-theoretical approaches can be unified by means of sextet polynomial.

%It is a wide-spread belief that
%Kekul\'{e} structures (or perfect matching in graph theory)
%are used only within resonance theory.
%Throughout this paper we shall consider only hexagonal systems
%with at least one Kekul\'{e} structure.
%All terms used but not defined can be found in \cite{GC89}.

Let $H$ be a hexagonal system with at least one Kekul\'{e} structure
(or perfect matching in graph theory).
A {\it Clar cover} of $H$ is a spanning subgraph of $H$
each (connected) component of which is either a hexagon or an edge.
A {\it resonant pattern} of $H$ is a set of hexagons of a Clar cover of $H$.
The {\it Clar number} $C(H)$ of $H$ is the maximum number of hexagons
in a resonant pattern of $H$.
The {\it sextet polynomial} of $H$ is defined as
$$\sg(H, x)=\sum_{k=0}^{C(H)}s(H, k)x^k,$$
where $s(H,k)$ denotes
the number of resonant patterns of $H$ having precisely $k$ hexagons
and $s(H,0)=1$.
%It is well known that $\sg(H,1)$ is
%precisely the Kekul\'e number $K(H)$ of $H$,
%i.e. the number of Kekul\'{e} structures of $H$.
%and $\sg'(B,1)=\sum_hK(B-h)$,
%where
%$B-h$ is the subgraph of $B$ obtained by
%deleting hexagon $h$ and all its adjacent edges,
%$h$ takes over all hexagons of $B$.

A hexagonal system $H$ is {\it cata-condensed} if
no three hexagons of $H$ have a common vertex,
and {\it peri-condensed} otherwise.
A cata-condensed hexagonal system $H$
is said to be {\it non-branched}
if every hexagon of $H$ has at most two neighbours.
In Figure \ref{TPC},
the triphenylene $T$ is cata-condensed,
the picene $P$ is non-branched,
and the coronene $C$ is peri-condensed.
We say that a hexagonal system $H$ is {\it thin}
if $H$ has no coronene $C$ as its nice subgraph
(a subgraph $C$ of $H$ is called {\it nice}
if either $H-C$ has a perfect matching or $H-C$ is empty).

\begin{figure}[ht]
  \centering
  \includegraphics[width=11cm]{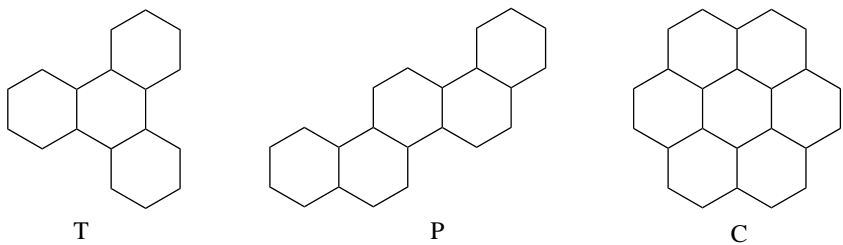}\\
  \caption{Triphenylene $T$, picene $P$ and coronene $C$}\label{TPC}
\end{figure}

Gutman \cite{Gut77} showed that the sextet polynomial of a non-branched cata-condensed hexagonal system
is precisely the matching polynomial of the corresponding Gutman tree.
He \cite{Gut82} also showed that the sextet polynomial of a resonant hexagonal system
coincides with the independence polynomial of the corresponding Clar graph.
Analytic properties of
the matching and independence polynomials of graphs have been extensively investigated.
For example,
Gutman and Godsil \cite{GG81} showed that zeros of the matching polynomial of a graph are all real.
Godsil \cite{God81} showed that the matching numbers of many graph sequences are approximately normally distributed.
Chudnovsky and Seymour \cite{CS07} showed that the independence polynomial of a clawfree graph has only real zeros,
Brown et al. \cite{BHN04} showed that real zeros of independence polynomials are dense in $(-\infty,0]$,
while complex zeros are dense in $\mathbb{C}$.
For the sextet polynomials of hexagonal systems,
Gutman \cite{Gut83}
asked for which hexagonal systems the
sextet polynomials have only real zeros.
Such zeros can be used for the calculation of the resonance energy and the aromaticity of hexagonal systems \cite{Aih77}.
Gutman \cite[Conjecture 44]{Gut83} also conjectured that every sextet polynomial has unimodal coefficients.
However, there is no systematic study of analytic properties of the sextet polynomials.

%A natural problem is to investigate analytic properties of the sextet polynomials.

As a tentative research,
we consider the sextet polynomials $P_n(x)$ of the pyrene chains $P_n$ (see Figure~\ref{Pn}).
%Clearly, $P_n$ is a branched, peri-condensed, resonant hexagonal system,
%and there are claws in the Clar graph of $P_n$ for $n\ge 2$.

\begin{figure}[ht]
  \centering
  \includegraphics[width=6.3cm]{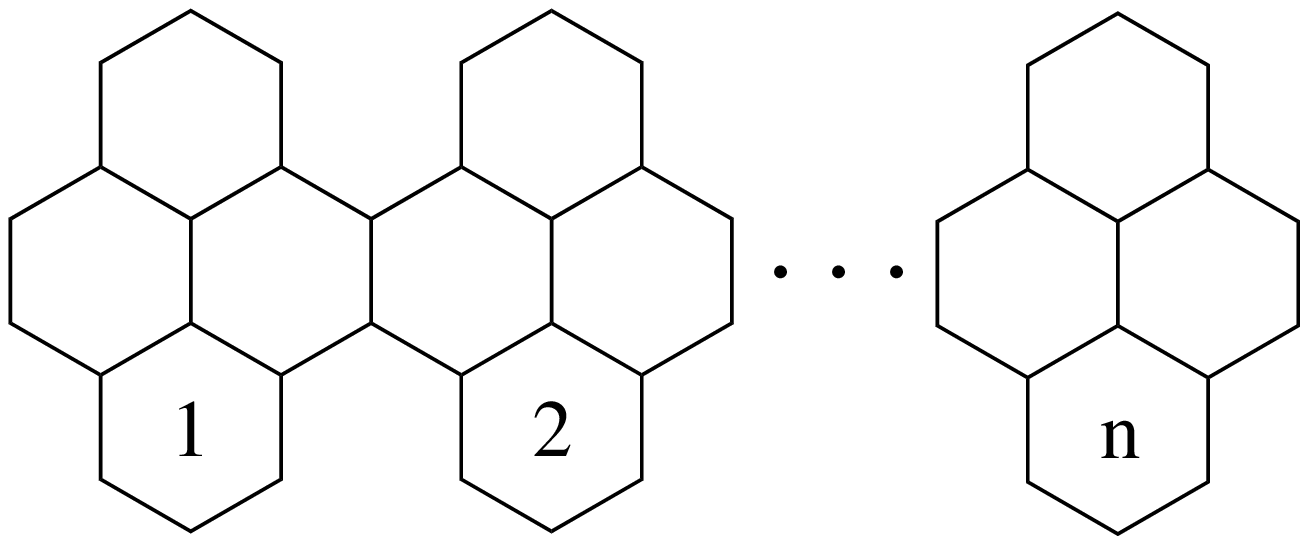}\\
  \caption{The pyrene chain $P_n$}\label{Pn}
\end{figure}

Define $P_0(x)=1$ for convenience.
Then
\begin{eqnarray*}
   P_1(x)&=&x^2+4x+1,\\
   P_2(x)&=&x^4+8x^3+17x^2+8x+1,  \\
   P_3(x) &=& %(x^2+4x+1)(x^4+8x^3+16x^2+8x+1)
   x^6+12x^5+49x^4+80x^3+49x^2+12x+1
\end{eqnarray*}
(see \cite{OH83} for instance).

The paper is organized as follows.
In the next section,
we show that zeros of all $P_n(x)$ are real,
located in the open interval $(-3-2\sqrt{2},-3+2\sqrt{2})$
and dense in the corresponding closed interval.
We also show that coefficients of $P_n(x)$ are
symmetric, unimodal, log-concave, and asymptotically normal
(by central and local limit theorems).
In \S 3, we provide more examples of hexagonal systems
whose sextet polynomials have only real zeros.
We show that real zeros of sextet polynomials of hexagonal systems are dense in the interval $(-\infty,0]$.
In the final section we suggest several open problems for further work.
In particular, we conjecture that
the sextet polynomial of a hexagonal system has log-concave coefficients.

Throughout this paper,
all terms used but not defined can be found in \cite{GC89}.

\section{Sextet polynomials of the pyrene chains}

In this section we devote to analytic aspects of the sextet polynomials $P_n(x)$ of the pyrene chain $P_n$.

Note first that $P_n$ are resonant hexagonal systems.
Hence the sextet polynomials $P_n(x)$ are precisely the independence polynomial of the corresponding Clar graph
(see Figure \ref{P-C}).

\begin{figure}[ht]
  \centering
  \includegraphics[width=6cm]{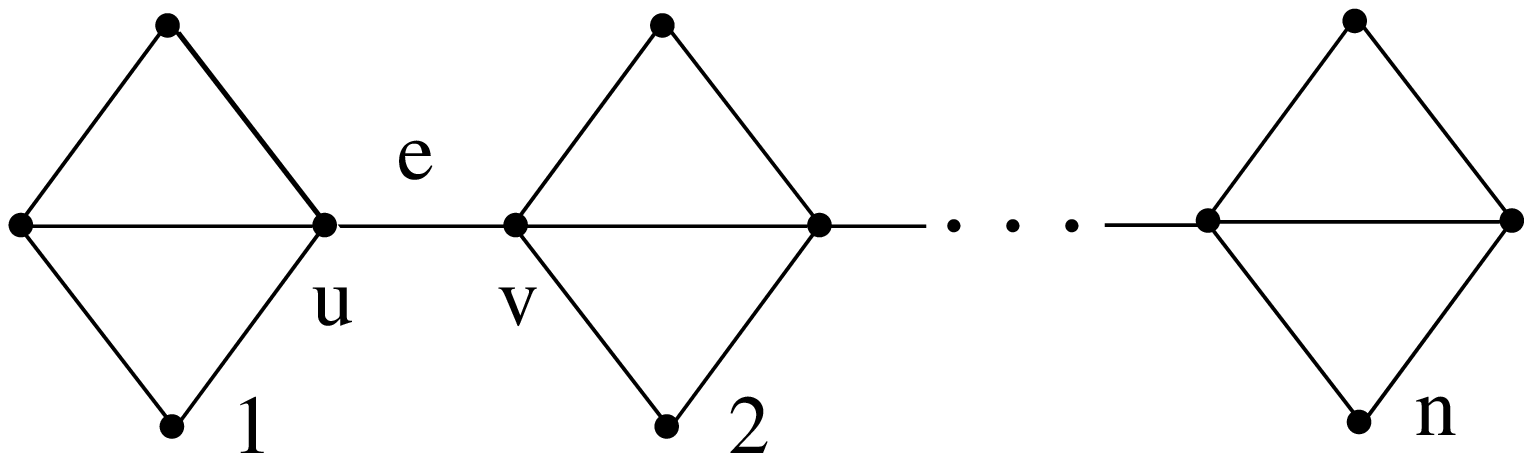}\\
  \caption{Clar graph of $P_n$}\label{P-C}
\end{figure}

Let $I(G,x)$ be the independence polynomial of a graph $G$.
Then for $uv\in E(G)$,
$$I(G,x)=I(G-uv,x)-x^2I(G-N(u)\cup N(v),x)$$
(see \cite[Theorem 3.9]{HL94} for instance).
It immediately follows that
\begin{equation}\label{d-rr}
P_n(x)=(x^2+4x+1)P_{n-1}(x)-x^2P_{n-2}(x)
\end{equation}
(see also \cite{OH83, WZ11}).

The following properties of $P_n(x)$ are immediate from the recurrence relation \eqref{d-rr}.

\begin{enumerate}[\rm (i)]
  \item %The largest number $C(P_n)$ of resonant sextets in $P_n$ is
  $\deg P_n(x)=2n$.
  \item $P_n(-1)=(n+1)(-1)^n$.
  \item $(x^2+4x+1)|P_{2n+1}(x)$ for $n\ge 0$.
\end{enumerate}

Further we may obtain various formulas of $P_n(x)$ from the recurrence relation \eqref{d-rr}.

\begin{prop}\label{pn-f}
\begin{enumerate}[\rm (i)]
  \item The sextet polynomials $P_n(x)$ have the generating function
\begin{equation}\label{d-gf}
\sum_{n\ge 0}P_n(x)y^n=\frac{1}{x^2y^2-(x^2+4x+1)y+1},
\end{equation}
  \item The sextet polynomials $P_n(x)$ have the Binet form
\begin{equation}\label{d-bf}
P_n(x)=\frac{\la_1^{n+1}-\la_2^{n+1}}{\la_1-\la_2},
\end{equation}
where
\begin{equation}\label{la12}
\la_{1,2}=\frac{x^2+4x+1\pm (x+1)\sqrt{x^2+6x+1}}{2}
\end{equation}
are two roots of the characteristic equation $\la^2-(x^2+4x+1)\la+x^2=0$.
%\begin{eqnarray}\label{d-bf}
%P_n(x) &=&\frac{1}{(x+1)\sqrt{x^2+6x+1}}
%\left[\left(\frac{x^2+4x+1+(x+1)\sqrt{x^2+6x+1}}{2}\right)^{n+1}\right.\nonumber\\
%&&-\left.\left(\frac{x^2+4x+1-(x+1)\sqrt{x^2+6x+1}}{2}\right)^{n+1}\right].
%\end{eqnarray}
  \item  The sextet polynomials $P_n(x)$ can be written as
\begin{equation}\label{d-gp}
P_n(x)=\sum_{i=0}^{n}\binom{2n+1-i}{i}x^i(x+1)^{2n-2i}.
\end{equation}
  \item The sextet polynomials $P_n(x)$ have the explicit expression
\begin{equation}\label{d-ee}
P_n(x)=\sum_{k=0}^{2n}\left[\sum_{i=0}^k\binom{2n+1-i}{i}\binom{2n-2i}{k-i}\right]x^k.
\end{equation}
In other words,
%The number of generalized Clar structures of $P_n$ having $k$ aromatic sextets is
\begin{equation}\label{spk}
s(P_n,k)=\sum_{i=0}^k\binom{2n+1-i}{i}\binom{2n-2i}{k-i},\quad k=0,1,2,\ldots,2n.
\end{equation}
\end{enumerate}
\end{prop}
\begin{proof}
The generating function \eqref{d-gf} and the Binet form \eqref{d-bf} of $P_n(x)$
can be obtained from the recurrence relation \eqref{d-rr} by standard combinatorial techniques
(see \cite{Wil06} for instance).
%so we omit their proof.
The formula \eqref{d-gp} can be obtained from the recurrence relation \eqref{d-rr} by induction on $n$,
and the expression \eqref{d-ee} follows immediately from \eqref{d-gp}.
\end{proof}

\begin{rem}
It is not difficult to give direct combinatorial interpretations of \eqref{d-gp} and \eqref{spk}.
\end{rem}

\begin{rem}
It is well known that $\sg(H,1)$ is
precisely the Kekul\'e number $K(H)$ of a thin hexagonal system $H$,
i.e. the number of Kekul\'{e} structures of $H$.
By means of Proposition \ref{pn-f},
we may give various expressions of
the Kekul\'e number of the pyrene chain $P_n$.
For example,
$$K(P_n)=\frac{1}{4\sqrt{2}}\left[(\sqrt{2}+1)^{2n+2}-(\sqrt{2}-1)^{2n+2}\right]
=\left\lfloor{\frac{(\sqrt{2}+1)^{2n+2}}{4\sqrt{2}}}\right\rfloor,$$
where $\lrf{r}$ denotes the largest integer no greater than the real number $r$.
\end{rem}

\begin{thm}\label{d-rz}
Zeros of the sextet polynomial $P_n(x)$ are real, distinct and in the open interval
$\left(-(\sqrt{2}+1)^2,-(\sqrt{2}-1)^2\right)$.
\end{thm}
\begin{proof}
Note that
$$\frac{x^2+4x+1\pm (x+1)\sqrt{x^2+6x+1}}{2}
=\left[\frac{x+1\pm \sqrt{x^2+6x+1}}{2}\right]^2.$$
Hence by \eqref{la12},
\begin{equation}\label{P-mu}
P_n(x)=\frac{\mu_1^{2n+2}-\mu_2^{2n+2}}{\mu_1^2-\mu_2^2},\quad
\mu_{1,2}=\frac{x+1\pm \sqrt{x^2+6x+1}}{2}.
\end{equation}
However,
$$\mu_1^{2n+2}-\mu_2^{2n+2}=(\mu_1^2-\mu_2^2)
\prod_{k=1}^n\left[(\mu_1+\mu_2)^2-4\mu_1\mu_2\cos^2\frac{k\pi}{2n+2}\right]$$
(see \cite{BC55} for instance).
Also, $\mu_1+\mu_2=x+1$ and $\mu_1\mu_2=-x$.
Hence by \eqref{P-mu},
\begin{eqnarray}\label{dne}
P_n(x)&=&\prod_{k=1}^n\left[(x+1)^2+4x\cos^2\frac{k\pi}{2n+2}\right]\nonumber\\
&=&\prod_{k=1}^n\left[x^2+2x\left(1+2\cos^2\frac{k\pi}{2n+2}\right)+1\right]\nonumber\\
&=&\prod_{k=1}^{n}\left(x+r_{k}\right)\left(x+\frac{1}{r_{k}}\right),
\end{eqnarray}
where
\begin{equation}\label{rk}
r_{k}=\left(\sqrt{1+\cos^2\frac{k\pi}{2n+2}}+\cos\frac{k\pi}{2n+2}\right)^2.
\end{equation}
Thus all zeros of $P_n(x)$ are $-r_1,\ldots,-r_n,-1/r_1,\ldots,-1/r_n$.

Clearly, $\cos\frac{k\pi}{2n+2}$ is strictly decreasing for $0\le k\le n+1$
and the function $f(x)=(\sqrt{1+x^2}+x)^2$ is strictly increasing since
$$f'(x)=\frac{2(\sqrt{1+x^2}+x)^2}{\sqrt{1+x^2}}>0,\qquad x\in\mathbb{R}.$$
Also, $r_0=(\sqrt{2}+1)^2$ and $r_{n+1}=1$.
Hence for $1\le k\le n$, all $r_k$ are distinct and $1<r_{k}<(\sqrt{2}+1)^2$.
It follows that all $1/r_k$ are also distinct and $(\sqrt{2}-1)^2<1/r_{k}<1$.
Thus
$$-(\sqrt{2}+1)^2<-r_1<\cdots<-r_n<-1<-1/r_n<\cdots<-1/r_1<-(\sqrt{2}-1)^2.$$
Zeros of $P_n(x)$ are therefore distinct real numbers in %the interval $\left(-(\sqrt{2}+1)^2,-(\sqrt{2}-1)^2\right)$.
\end{proof}

\begin{rem}
From the explicit expression \eqref{rk}
it is easy to see that zeros of all $P_n(x)$ are dense in the closed interval
$\left[-(\sqrt{2}+1)^2, -(\sqrt{2}-1)^2\right]$.
See Example \ref{P-d} for a stronger result.
\end{rem}

\begin{rem}
Let
\begin{equation}\label{D-mu}
D_n(x)=\frac{\mu_1^{n+1}-\mu_2^{n+1}}{\mu_1-\mu_2},
\end{equation}
where $\mu_1+\mu_2=x+1$ and $\mu_1\mu_2=-x$.
Then $D_0(x)=1, D_1(x)=x+1$ and
\begin{equation}\label{D-rr}
D_n(x)=(x+1)D_{n-1}(x)+xD_{n-2}(x),\quad n\ge 2.
\end{equation}
By \eqref{P-mu},
\begin{equation}\label{P-D}
P_n(x)=\frac{D_{2n+1}(x)}{x+1}.
\end{equation}
It is worth noting that $D_n(x)$ are the Delannoy polynomials
(see \cite{WZC19} for instance).
\end{rem}

\begin{rem}
Let $H$ be a hexagonal system whose sextet polynomial has only real zeros.
Suppose that $\sg(H,x)=\prod_k(x+r_k)$.
Then {\it Aihara's resonant energy} of $H$ is
$$RE^*(H)=\sum_{k}1/\sqrt{r_k}$$
(see \cite{Aih77} or \cite[\S 7.5.1]{GC89} for instance).
It follows from \eqref{dne} that
$$RE^*(P_n)=2\sum_{k=1}^n\sqrt{1+\cos^2\frac{k\pi}{2n+2}}
\rightarrow \frac{4n}{\pi}\int_{0}^{\frac{\pi}{2}}\sqrt{1+\cos^2\theta}d\theta
\approx 2.432n.$$
\end{rem}

%\subsection{Coefficients of $P_n(x)$}

In what follows we pay our attention to coefficients $s(P_n,k)$ of the sextet polynomials $P_n(x)$.

Let $a_0,a_1,\ldots,a_n$ be a finite sequence of nonnegative numbers.
We say that the sequence is {\it symmetric} if $a_k=a_{n-k}$ for $0\le k\le n$.
We say that the sequence is {\it log-concave} if $a_{k-1}a_{k+1}\le a_k^2$ for $0<k<n$,
and {\it unimodal} if
$$a_0\le a_1\le\cdots\le a_m\ge a_{m+1}\ge\cdots\ge a_n$$
for some $m$.
Clearly, a log-concave sequence of positive numbers is unimodal.
A classical approach for attacking the unimodality and log-concavity problem
is to use the following.

\begin{NI}
Suppose that the real polynomial $f(x)=\sum_{k=0}^na_kx^k$ has only real zeros. Then
$$a_k^2\ge a_{k-1}a_{k+1}\frac{(k+1)(n-k+1)}{k(n-k)},\qquad k=1,2,\ldots,n-1.$$
If all $a_k$'s are nonnegative,
then the sequence $a_0,a_1,\ldots,a_n$ of coefficients is log-concave and unimodal.
\end{NI}

For convenience, we say that a polynomial is symmetric (resp., unimodal, log-concave)
if the sequence of its coefficients has such a property.
Let $f(x)$ be a polynomial of degree $n$ with nonzero constant term.
It is claer that $f(x)$ is symmetric if and only if $x^nf(1/x)=f(x)$.
%We refer the reader to \cite{Bre89} for details.

\begin{thm}
The sextet polynomial $P_n(x)$ has symmetric, unimodal and log-concave coefficients.
\end{thm}
\begin{proof}
The symmetry is easily obtained from \eqref{spk} or \eqref{dne}.
%from the recursive \eqref{d-rr} by induction on $n$,
%or more directly, from the explicit expression \eqref{spk}.
The unimodality and log-concavity follows immediately from the Newton inequality.
\end{proof}

Let $a(n,k)$ be a double-indexed sequence of nonnegative numbers and let
\begin{equation*}\label{pnk}
p(n,k)=\frac{a(n,k)}{\sum_{j}a(n,j)}
\end{equation*}
denote the normalized probabilities.
Following Bender~\cite{Ben73},
we say that the sequence $a(n,k)$ is {\it asymptotically normal by a central limit theorem},
if
\begin{equation}\label{clt}
\lim_{n\rightarrow\infty}\sup_{x\in\mathbb{R}}\left|\sum_{k\le\mu_n+x\sigma_n}p(n,k)-\frac{1}{\sqrt{2\pi}}\int_{-\infty}^xe^{-t^2/2}dt\right|=0,
\end{equation}
where $\mu_n$ and $\sigma^2_n$ are the mean and variance of $a(n,k)$, respectively.
We say that $a(n,k)$ is {\it asymptotically normal by a local limit theorem} on $\mathbb{R}$ if
\begin{equation}\label{llt}
\lim_{n\rightarrow\infty}\sup_{x\in\mathbb{R}}\left|\sigma_np(n,\lfloor\mu_n+x\sigma_n\rfloor)-\frac{1}{\sqrt{2\pi}}e^{-x^2/2}\right|=0.
\end{equation}
In this case,
\begin{equation*}\label{asy}
a(n,k)\sim \frac{e^{-x^2/2}\sum_{j}a(n,j)}{\sigma_n\sqrt{2\pi}} \textrm{ as } n\rightarrow \infty,
\end{equation*}
where $k=\mu_n+x\sigma_n$ and $x=O(1)$.
Clearly, the validity of \eqref{llt} implies that of \eqref{clt}.

Many well-known combinatorial sequences enjoy central and local limit theorems.
For example,
the matching numbers of many graph sequences are approximately normally distributed \cite{God81}.
%See \cite{Can15} for an excellent survey about asymptotic normality of combinatorial sequences.
A standard approach to demonstrating asymptotic normality is the following criterion.
%(see \cite[Theorem 2]{Ben73} for instance and \cite[Example 3.4.2]{Can15} for historical remarks).

\begin{lem}[{\cite[Theorem 2]{Ben73}}]\label{lem-rzv}
Suppose that $A_n(x)=\sum_{k=0}^na(n,k)x^k$ have only real zeros and $A_n(x)=\prod_{i=1}^n(x+r_i)$,
where all $a(n,k)$ and $r_i$ are nonnegative.
Let $$\mu_n=\sum_{i=1}^n\frac{1}{1+r_i}$$
and $$\sigma^2_n=\sum_{i=1}^n\frac{r_i}{(1+r_i)^2}.$$
Then if $\sigma_n^2\rightarrow+\infty$,
the numbers $a(n,k)$ are asymptotically normal (by central and local limit theorems)
with the mean $\mu_n$ and variance $\sigma_n^2$.
\end{lem}

\begin{thm}\label{dnk-an}
The resonant sextet numbers $s(P_n,k)$ are asymptotically normal (by central and local limit theorems)
with the mean $\mu_n=n$ and variance $\sigma_n^2\sim\sqrt{2}n/4$.
\end{thm}
\begin{proof}
The mean $\mu_n=n$ is clear by the symmetry of $s(P_n,k)=s(P_n,2n-k)$.
On the other hand,
since
$$\sum_{k=1}^n\frac{1/r_k}{(1+1/r_k)^2}=\sum_{k=1}^n\frac{r_k}{(1+r_k)^2},$$
we have by \eqref{dne}
$$\sigma_n^2=2\sum_{k=1}^n\frac{\left(\cos\frac{k\pi}{2n+2}+\sqrt{1+\cos^2\frac{k\pi}{2n+2}}\right)^2}
{\left[1+\left(\cos\frac{k\pi}{2n+2}+\sqrt{1+\cos^2\frac{k\pi}{2n+2}}\right)^2\right]^2}.$$
Note that
$$1+\left(c+\sqrt{1+c^2}\right)^2
=2\left(1+c^2+c\sqrt{1+c^2}\right)
=2\left(c+\sqrt{1+c^2}\right)\sqrt{1+c^2}.$$
Hence
$$\sigma_n^2=2\sum_{k=1}^n\frac{1}{4\left(1+\cos^2\frac{k\pi}{2n+2}\right)}
=\frac{1}{2}\sum_{k=1}^n\frac{1}{\left(1+\cos^2\frac{k}{n+1}\frac{\pi}{2}\right)}.$$
Thus
$$\sigma_n^2
\rightarrow
\frac{n}{\pi}\int_{0}^{\pi/2}\frac{1}{1+\cos^2\theta}d\theta
=\frac{n}{\pi}\left[-\frac{1}{\sqrt{2}}\arctan(\sqrt{2}\cot\theta)\right]^{\frac{\pi}{2}}_{0}
=\frac{\sqrt{2}}{4}n.$$
The statement follows from Lemma \ref{lem-rzv}.
\end{proof}

\section{Sextet polynomials with only real zeros}

In this section
we present more examples of hexagonal systems whose sextet polynomials have only real zeros.
We call $r$ a {\it sextet zero} if $r$ is one zero of the sextet polynomial of a hexagonal system.
%A zero of the sextet polynomial of a hexagonal system is called a {\it sextet zero}.
We will show that real sextet zeros are dense in the interval $(-\infty,0]$.
%We first recall some criteria for polynomial with only real zeros.

%Let $\rz$ denote the set of real polynomials with only real zeros.
%For $f\in\rz$ and $\deg f=n$,
Let $f(x)$ be a real polynomial with only real zeros and $\deg f=n$.
Denote its zeros by $r_n(f)\le\cdots\le r_2(f)\le r_1(f)$.
Given two real polynomials $f$ and $g$ with only real zeros,
we say that $g$ {\it interlaces} $f$, denoted by $g\sp f$,
if $\deg f=\deg g+1$ and
$$r_{n}(f)\le r_{n-1}(g)\le r_{n-1}(f)\le\cdots\le r_2(f)\le r_1(g)\le r_1(f),$$
or $\deg f=\deg g$ and
$$r_{n}(g)\le  r_{n}(f)\le r_{n-1}(g)\le r_{n-1}(f)\le\cdots\le r_2(f)\le r_1(g)\le r_1(f).$$

\begin{lem}[{\cite[Theorem 2.3]{LW07rz}}]\label{lem-lw}
Let $F(x)=a(x)f(x)+b_1(x)g_1(x)+\cdots+b_k(x)g_k(x)$,
where $F,f, a, b_j, g_j$ are real polynomials.
Suppose that
\begin{enumerate}[\rm (i)]
\item
%$f,g_j\in\rz$ and $g_j\sp f$ for each $j$;
$f,g_j$ have only real zeros and $g_j\sp f$ for each $j$;
\item
$g_1,\ldots,g_k$ and $F$ have positive leading coefficients;
\item
$\deg f\le\deg F\le\deg f+1$.
\end{enumerate}
If $b_j(r)\le 0$ for each $j$ whenever $f(r)=0$, then $F$ has only real zeros and $f\sp F$.
\end{lem}

Let $G$ be a graph and $M(G;x)$ the matching polynomial of $G$.
Then for arbitrary $v\in V(G)$,
$$M(G;x)=xM(G-v;x)-\sum_{u\sim v}M(G-\{u,v\};x).$$
Thus $M(G;x)$ has only real zeros and $M(G-v;x)\sp M(G;x)$ by induction and by Lemma \ref{lem-lw}.
%The sextet polynomial of a non-branched cata-condensed hexagonal system
%is precisely the matching polynomial of the corresponding Gutman tree
%and such a polynomial has only real zeros.

The following corollary is an immediate consequence of Lemma \ref{lem-lw}.
%Since the sextet polynomials of hexagonal systems often satisfy various recurrence relations
%(see \cite{OH83,OH84} for instance),
%the following criterion, which is a special case of \cite[Theorem 2.3]{LW07rz},
%will be very useful.

\begin{coro}\label{c-3rr}
Let $(f_n(x))_{n\ge 0}$ be a sequence of polynomials of positive coefficients
with $\deg f_0=0$ and $\deg f_{n-1}\le \deg f_n\le \deg f_{n-1}+1$.
Suppose that
\begin{equation}\label{3rr}
  f_n(x)=a(x)f_{n-1}(x)+b(x)f_{n-2}(x),\quad n=2,3,\ldots.
\end{equation}
If $b(x)\le 0$ for $x\le 0$,
then all $f_n(x)$ have only real zeros and $f_{n-1}\sp f_n$.
\end{coro}

\begin{exm}
Recall that the sextet polynomials $P_n(x)$ of the pyrene chain satisfy
$P_n(x)=D_{2n+1}(x)/(x+1)$,
where $D_n$ %are the Delannoy polynomials
satisfy the recurrence relation
$D_n(x)=(x+1)D_{n-1}(x)+xD_{n-2}(x)$.
From Corollary \ref{c-3rr} it follows that all $D_n(x)$ have only real zeros,
and so do all $P_n(x)$.
\end{exm}

\begin{figure}[h]
  \centering
  \includegraphics[width=13cm]{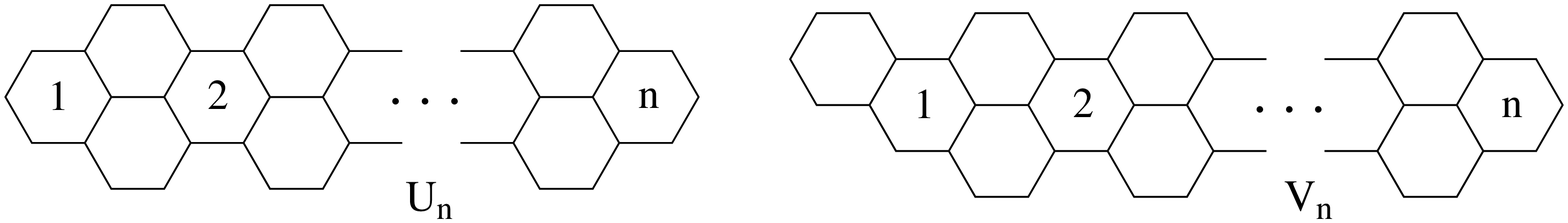}\\
  \caption{$U_n$ and $V_n$}\label{UV}
\end{figure}

\begin{exm}
Let $U_n$ and $V_n$ be hexagonal systems as shown in Figure \ref{UV}.
Let $U_n(x)$ and $V_n(x)$ be the sextet polynomials corresponding to $U_n$ and $V_n$.
Then
%both $(U_n(x))_{n\ge 0}$ and $(V_n(x))_{n\ge 0}$ satisfy the same recurrence relation
%$$f_n(x)=(2x+1)f_{n-1}(x)+x(1-x)f_{n-2}(x)$$
%with $U_0(x)=V_0(x)=1$, $U_1(x)=x+1$ and $V_1(x)=1+2x$
$$U_n(x)=(2x+1)U_{n-1}(x)+x(1-x)U_{n-2}(x),\quad U_0(x)=1,\quad U_1(x)=x+1$$
and
$$V_n(x)=(2x+1)V_{n-1}(x)+x(1-x)V_{n-2}(x),\quad V_0(x)=1,\quad V_1(x)=2x+1$$
(see \cite{OH84} for instance).
Thus all $U_n(x)$ and $V_n(x)$ have only real zeros by Corollary \ref{c-3rr}.
\end{exm}

\begin{exm}
Let $L_n$ be the line hexagonal chains with $n$ hexagons
shown as in Figure \ref{Ln}.
Then it is clear that the corresponding sextet polynomial is $L_n(x)=nx+1$.

\begin{figure}[h]
\centering
  \includegraphics[width=5.7cm]{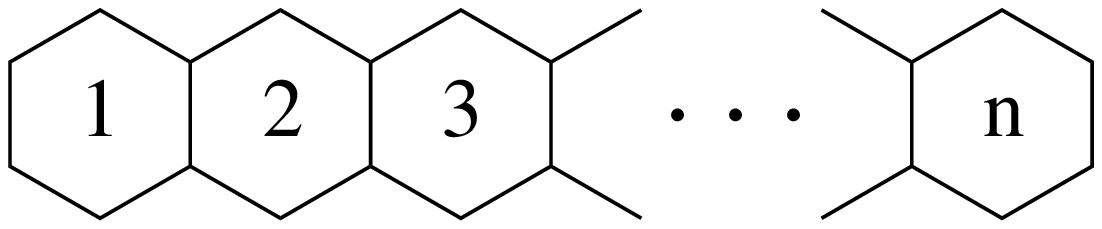}\\
  \caption{$L_n$}\label{Ln}
\end{figure}

For $m\ge 2$, let $L^{(m)}_n$ be the hexagonal chains shown as in Figure~\ref{Lmn}
and let $L^{(m)}_n(x)$ be the corresponding sextet polynomials.
By distinguishing two cases whether the hexagon $h$ is an aromatic sextet,
%By means of the method established in \cite{Gutetal}, it is easy to obtain the recurrence relation
%by \eqref{cc-rr},
we obtain
\begin{equation}\label{Lmn-rr}
L^{(m)}_n(x)=[(m-2)x+1]L^{(m)}_{n-1}(x)+xL^{(m)}_{n-2}(x)
\end{equation}
with $L^{(m)}_0(x)=1$ and $L^{(m)}_1(x)=(m-1)x+1$
(see also \cite{Gutetal}).
It follows that $L^{(m)}_n(x)$ have only real zeros from Corollary \ref{c-3rr}.
\begin{figure}[h]
  \centering
  \includegraphics[width=12.7cm]{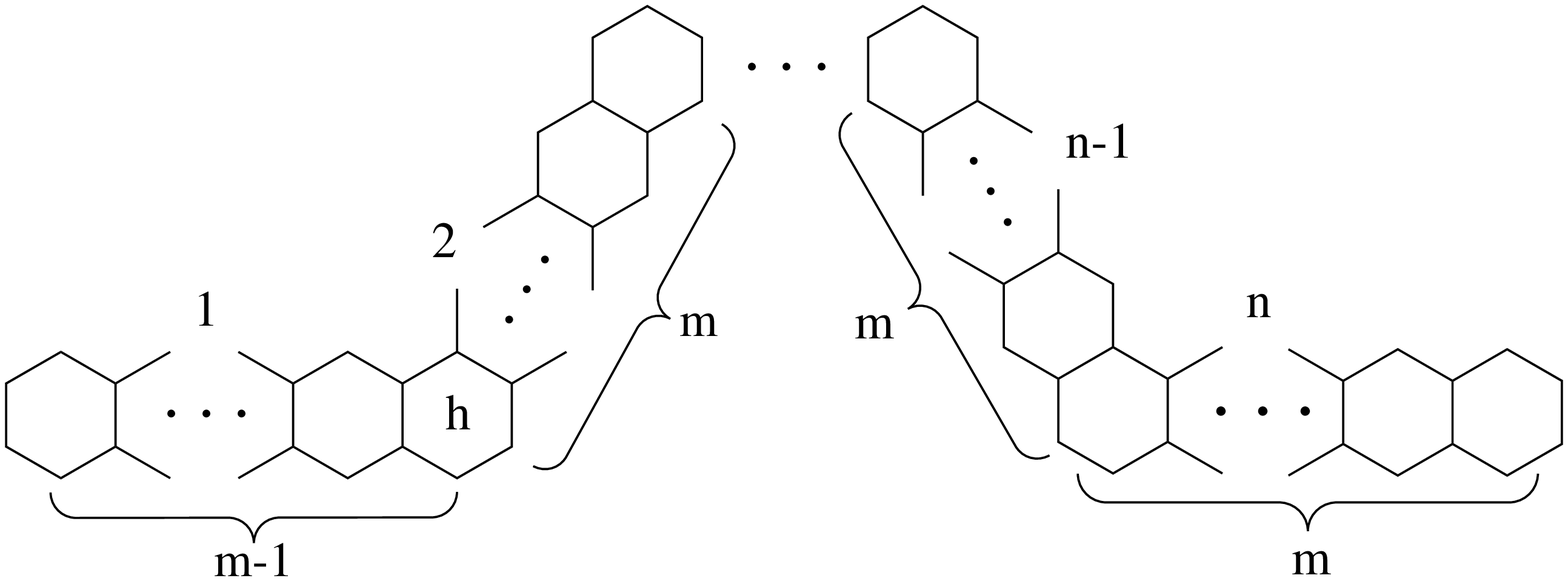}\\
  \caption{$L^{(m)}_n$}\label{Lmn}
\end{figure}
\end{exm}

We next discuss the location of zeros of all $L^{(m)}_n(x)$.
We need the following lemma.

Let $(f_n(x))_{n\ge 0}$ be a sequence of complex polynomials.
We say that the complex number $x$ is a {\it limit of zeros} of the sequence $(f_n(x))_{n\ge 0}$
if there is a sequence $(z_n)_{n\ge 0}$ such that $f_n(z_n)=0$ and $z_n\rightarrow x$ as $n\rightarrow +\infty$.
Suppose now that $(f_n(x))_{n\ge 0}$ is a sequence of polynomials satisfying the recursion %of degree $k$
$$f_{n+k}(x)=-\sum_{j=1}^kc_j(x)f_{n+k-j}(x)$$
where $c_j(x)$ are polynomials in $x$.
Let $\la_j(x)$ be all zeros of the associated characteristic equation
$\la^k+\sum_{j=1}^{k}c_j(x)\la^{k-j}=0$.
It is well known that if $\la_j(x)$ are distinct, then
\begin{equation}\label{r}
f_n(x)=\sum_{j=1}^k\alpha_j(x)\la_j^n(x),
\end{equation}
where $\alpha_j(x)$ are determined from the initial conditions.

\begin{lem}[{\cite[Theorem]{BKW78}}]\label{bkw}
Under the non-degeneracy requirements that in \eqref{r}
no $\alpha_j(x)$ is identically zero
and that for no pair $i\neq j$ is $\la_i(x)\equiv\omega\la_j(x)$ for some $\omega\in\mathbb{C}$ of unit modulus,
then $x$ is a limit of zeros of $(f_n(x))_{n\ge 0}$ if and only if either
\begin{enumerate}[\rm (i)]
  \item two or more of the $\la_i(x)$ are of equal modulus, and strictly greater (in modulus) than the others; or
  \item for some $j$, $\la_j(x)$ has modulus strictly greater than all the other $\la_i(x)$ have, and $\alpha_j(x)=0$.
\end{enumerate}
\end{lem}

\begin{exm}\label{P-d}
Recall that the sextet polynomials of the pyrene chains
$$P_n(x)=\frac{\la_1^{n+1}-\la_2^{n+1}}{\la_1-\la_2},\quad
\la_{1,2}=\frac{x^2+4x+1\pm (x+1)\sqrt{x^2+6x+1}}{2}.$$
We apply Lemma \ref{bkw} to show that
each $x\in\left[-(\sqrt{2}+1)^2,-(\sqrt{2}-1)^2\right]$ is a limit of zeros of the sequence $(P_n(x))_{n\ge 0}$.

The non-degeneracy conditions of Lemma \ref{bkw} are clearly satisfied.
So the limits of zeros of $(P_n(x))_{n\ge 0}$ are those $x$ for which $|\la_1|=|\la_2|$,
i.e.,
$$|x^2+4x+1+(x+1)\sqrt{x^2+6x+1}|=|x^2+4x+1-(x+1)\sqrt{x^2+6x+1}|.$$
In other words, $\sqrt{x^2+6x+1}$ must be a pure imaginary or zero.
Thus $x^2+6x+1\le 0$, i.e., $$-3-2\sqrt{2}\le x\le -3+2\sqrt{2},$$
which is what we wanted to show.

As a result,
zeros of all $P_n(x)$ are dense in %the closed interval
$\left[-(\sqrt{2}+1)^2,-(\sqrt{2}-1)^2\right]$.
\end{exm}

\begin{thm}\label{sp-zd}
Real sextet zeros of hexagonal systems are dense in the interval $(-\infty,0]$.
\end{thm}
\begin{proof}
It suffices to show that real sextet zeros of hexagonal systems $L^{(m)}_n$ for $m\ge 2$
are dense in the interval $(-\infty,0]$.

For $m=2$, we have
\begin{equation}\label{I2}
L^{(2)}_n(x)=L^{(2)}_{n-1}(x)+xL^{(2)}_{n-2}(x)
\end{equation}
with $L^{(2)}_0(x)=1$ and $L^{(2)}_1(x)=1+x$.
Solve \eqref{I2} to obtain the Binet form
\begin{equation}\label{I2-bf}
L^{(2)}_n(x)=
%\frac{\la_1^{n+1}(x)-\la_1^{n+1}(x)}{\la_1(x)-\la_2(x)},\quad
%\la_{1,2}(x)=\frac{1\pm\sqrt{1+4x}}{2}.
\frac{\left(1+\sqrt{1+4x}\right)^{n+2}-\left(1-\sqrt{1+4x}\right)^{n+2}}{2^{n+2}\sqrt{1+4x}}.
\end{equation}
It follows from Lemma \ref{bkw} that
$x$ is a limits of zeros of $(L^{(2)}_n(x))_{n\ge 0}$ if and only if
$$|1+\sqrt{1+4x}|=|1-\sqrt{1+4x}|,$$
i.e., $1+4x\le 0$, or equivalently, $x\in (-\infty,-1/4]$.
In particular,
zeros of $L^{(2)}_n(x)$ are dense in the interval $(-\infty,-1/4]$.

For $m\ge 3$,
solve \eqref{Lmn-rr} to obtain the Binet form
$$L^{(m)}_n(x)=c_1(x)\la_1^n(x)+c_2(x)\la_2^n(x),$$
where
$$
\left\{
  \begin{array}{l}
    c_1(x)=\dfrac{\sqrt{(m-2)^2x^2+2mx+1}+mx+1}{2\sqrt{(m-2)^2x^2+2mx+1}} \\
    c_2(x)=\dfrac{\sqrt{(m-2)^2x^2+2mx+1}-mx-1}{2\sqrt{(m-2)^2x^2+2mx+1}}
  \end{array}
\right.
$$
and
$$
\left\{
  \begin{array}{l}
    \la_1(x)=\dfrac{(m-2)x+1+\sqrt{(m-2)^2x^2+2mx+1}}{2} \\
    \la_2(x)=\dfrac{(m-2)x+1-\sqrt{(m-2)^2x^2+2mx+1}}{2}.
  \end{array}
\right.
$$
Similarly, $x$ is a limit of zeros of $L^{(m)}_n(x)$ if and only if
$(m-2)^2x^2+2mx+1\le 0$, i.e., $x\in I_m=[a_m,b_m]$, where
\begin{equation}\label{Im}
a_m=-\frac{1}{(\sqrt{m-1}-1)^2},\quad
b_m=-\frac{1}{(\sqrt{m-1}+1)^2}.
  %I^{(m)}=\left[-\frac{1}{(\sqrt{m-1}-1)^2},-\frac{1}{(\sqrt{m-1}+1)^2}\right].
\end{equation}
Clearly,
$(a_m), (b_m)$ are increasing sequences
and $a_{m+1}<b_{m}$ for $m\ge 3$.
Hence $I_m\cap I_{m+1}\neq\emptyset$.
Also, $a_3=-(\sqrt{2}+1)^2$ and $b_m\rightarrow 0$ when $m\rightarrow +\infty$.
Thus zeros of $L^{(m)}_n(x)$ for $m\ge 3$ are dense in the interval $[-(\sqrt{2}+1)^2,0]$.

Since $-(\sqrt{2}+1)^2<-1/4$,
we conclude that zeros of $L^{(m)}_n(x)$ for $m\ge 2$ are dense in $(-\infty,0]$,
as desired.
\end{proof}

Note that $L^{(m)}_n$ are non-branched cata-condensed hexagonal systems,
which imply that $L^{(m)}_n(x)$ are the matching polynomials
of the corresponding Gutman trees.
Hence the following folklore result is immediate.

\begin{prop}
Zeros of the matching polynomials of graphs are dense in %the interval
$(-\infty,0]$.
\end{prop}

\section{Further work}

The sextet polynomial of a hexagonal system may have non-real zeros.
For example, the sextet polynomial of the triphenylene $T$
is $\sg(T,x)=x^3+3x^2+4x+1$,
whose three zeros are
$$x_1\approx -0.32,\quad x_{2,3}\approx -1.34\pm 1.16i.$$

By checking all hexagonal systems with at most $5$ hexagons,
we found that every sextet polynomial has a real sextet zero in the interval $[-1,0)$ and
all zeros in the left-half plane.
It is possible that all sextet polynomials have such properties.
Gutman \cite{Gut83} conjectured that every sextet polynomial has unimodal coefficients.
We propose the following stronger conjecture.

\begin{conj}\label{sp-lc}
The sextet polynomial of a hexagonal system has log-concave coefficients.
\end{conj}

In 1996, H. Zhang and F. Zhang \cite{ZZ96a} introduced the concept of
Clar covering polynomial $\chi(H,x)$ of a hexagonal system $H$:
%as a counting polynomial of Clar covers of a benzenoid hydrocarbon.
%which includes many information about resonant structures, such as
%the Clar number, Clar structure count, Kekul\'e structure count and the first Herndon number.
$$\chi(H,x)=\sum_{k=0}^{C(H)}c(H,k)x^k,$$
where $c(H,k)$ denote the number of Clar covers of $H$ having precisely $k$ hexagons.
They \cite{ZZ00} showed that the coefficients of the first half terms with higher degrees
in a Clar covering polynomial form a strictly decreasing sequence,
and further proposed the following unimodal conjecture.

\begin{conj}[{\cite[Conjecture 8]{ZZ00}}]\label{ccp-u}
The Clar covering polynomial of a hexagonal system has unimodal coefficients.
\end{conj}

Numerical results suggest the following log-concave conjecture.

\begin{conj}\label{ccp-lc}
The Clar covering polynomial of a hexagonal system has log-concave coefficients.
\end{conj}

Let $p(H,k)$ denote the number of perfect matchings of a hexagonal system $H$
which contains precisely $k$ proper sextets.
Define the polynomial
$$\phi(H,x)=\sum_{k=0}^{C(H)}p(H,k)x^k.$$
%We say that a hexagonal system $H$ is {\it thin}
%if $H$ has no coronene $C$ as its nice subgraph
%(a subgraph $C$ of $H$ is called {\it nice}
%if either $H-C$ has a perfect matching or $H-C$ is empty).
Zhang and Zhang \cite[Theorem 11]{ZZ00}
showed that $\phi(H,x)=\sg(H,x)$ if %and only if
$H$ is a thin hexagonal system.
%$B$ has no coronene as its nice subgraph
We propose the following conjecture.

\begin{conj}\label{conj-p}
Let $H$ be a Kekul\'ean hexagonal system.
Then the polynomial $\phi(H,x)$ has log-concave coefficients.
%The polynomial $\phi(B,x)$ of a Kekul\'ean hexagonal system $B$ is log-concave.
\end{conj}

Zhang and Zhang \cite[Theorem 2]{ZZ00} showed that
$\chi(H,x)=\phi(H,x+1)$
for every Kekul\'{e} hexagonal system $H$.
It is known \cite[Theorem 2]{Hog74} that
if a polynomial $f(x)$ with positive coefficients is log-concave,
then so is the polynomial $f(x+1)$.
Thus we conclude that if Conjecture \ref{conj-p} is true,
then so is Conjecture \ref{ccp-lc}.

It is known \cite[Theorem 2.8]{ZSS13} that if $H$ is a Kekul\'ean hexagonal system,
then $\chi(H,x)=C(R(H),x)$,
where $R(H)$ is the resonance graph of $H$
and $C(G,x)$ is the cube polynomial of a median graph $G$.
Bre\v{s}ar et al. \cite[Theorem 5.1]{BKS06} showed that
$C(G,x)$ has a real zero in the interval $[-2,-1)$.
It follows that the sextet polynomial of a thin hexagonal system has a real zero in the interval $[-1,0)$.
It is known that there is no cube zeros in the interval $[-1,+\infty)$
and there exists arbitrarily small negative real cube zeros
\cite[Corollary 3.2 and Proposition 5.2]{BKS06}.
Actually,
real cube zeros are dense in the interval $(-\infty,-1]$
from the proof of Theorem \ref{sp-zd}.

\section*{Acknowledgement}

This work was supported partially by the National Natural Science Foundation of China (Nos. 11771065, 11871304),
the Natural Science Foundation of Shandong Province of China (No. ZR2017MA025).
%The authors thank the anonymous referee for his/her careful reading and kind comments.

%\section*{References}

\end{document}